\documentclass{amsart}
\usepackage{amssymb,amscd,verbatim, amsthm,amsmath,amsgen,xspace}

\usepackage[normalem]{ulem}

\usepackage{graphicx,color}
\usepackage{amsmath,amsfonts,amssymb}
\usepackage[latin2]{inputenc}                                          

\def\softd{{\leavevmode\setbox1=\hbox{d}%
\hbox to 1.05\wd1{d\kern-0.4ex{\char039}\hss}}}
\def\softt{{\leavevmode\setbox1=\hbox{t}
\hbox to \wd1{t\kern-0.6ex{\char039}\hss}}}
\def\softl{l\kern-0.45ex\raise0.1ex\hbox{''}\kern-0.10ex}
\def\softL{L\kern-0.8ex\raise0.1ex\hbox{''}\kern0.1ex}

\numberwithin{equation}{section}
\newtheorem{theorem}{Theorem}[section]
\newtheorem{definition}[theorem]{Definition}

\newtheorem{lemma}[theorem]{Lemma}

\newtheorem{proposition}[theorem]{Proposition}
\newtheorem{remark}[theorem]{Remark}

\newcommand{\mN}{\mathbb N}
\newcommand{\be}{\begin{eqnarray}}
\newcommand{\ee}{\end{eqnarray}}
\newcommand{\bd}{\begin{definition}}
\newcommand{\ed}{\end{definition}}
\newcommand{\br}{\begin{remark}}
\newcommand{\er}{\end{remark}}

\newcommand{\bt}{\begin{tabular}}
\newcommand{\et}{\end{tabular}}

\newcommand{\bl}{\begin{lemma}}
\newcommand{\el}{\end{lemma}}

\newcommand{\bp}{\begin{picture}}
\newcommand{\ep}{\end{picture}}
\newcommand{\bi}{\begin{itemize}}
\newcommand{\ei}{\end{itemize}}
\newcommand{\bq}{\begin{quotation}}
\newcommand{\eq}{\end{quotation}}

\input epsf
\input xy
\xyoption{all}

\def\id{\mathop{\hbox{id}}\nolimits}

\def\Cas{\mathop{\hbox {Cas}}\nolimits}

\def\mfrak{\mathfrak}

\newcommand{\twedge}{\textstyle{\bigwedge}}

\newcommand{\pf}{\begin{proof}}
\newcommand{\epf}{\end{proof}}
\newcommand{\eeq}{\end{equation}}
\newcommand{\eqn}{\begin{equation*}}
\newcommand{\eeqn}{\end{equation*}}
\newcommand\bed{\begin{definition}}
\newcommand\ebed{\end{definition}}
\newcommand\bethm{\begin{theorem}}
\newcommand\ebethm{\end{theorem}}
\newcommand\bealigned{\begin{aligned}}
\newcommand\ebealigned{\end{aligned}}

\newcommand{\frg}{\mathfrak{g}}

\newcommand{\frk}{\mathfrak{k}}

\newcommand{\frp}{\mathfrak{p}}

\newcommand{\frgl}{\mathfrak{gl}}
\newcommand{\frsl}{\mathfrak{sl}}

\newcommand{\frso}{\mathfrak{so}}

\newcommand{\bbC}{\mathbb{C}}

\newcommand{\uqsl}{U_q(\mathfrak{sl}_2)}

\begin{document}

\baselineskip13pt

\title{Braided coproduct, antipode and adjoint action for $U_q(sl_2)$}

\author{Pavle Pand\v{z}i\'{c}}
\address{Department of Mathematics, Faculty of Science, University of Zagreb, Bijeni\v{c}ka cesta 30, 10000 Zagreb, Croatia}
\email{pandzic@math.hr}

\author{Petr Somberg}
\address[Somberg]{Mathematical Institute MFF UK\\
Sokolovsk\'a 83, 18000 Praha 8 - Karl\'{\i}n, Czech Republic}
\email{somberg@karlin.mff.cuni.cz}

\thanks{
 P.~Pand\v{z}i\'{c} was supported by the QuantiXLie Centre of Excellence, a project cofinanced by the Croatian
  Government and European Union through the European Regional Development Fund - the Competitiveness and
  Cohesion Operational Programme (KK.01.1.1.01.0004). 
}
\thanks{  P.~Somberg was supported by grant GA\v CR 22-00091S
}

\date {}

\begin{abstract} 
	Motivated by our attempts to construct an analogue of the Dirac operator in the setting of $U_q(\mathfrak{sl}_n)$, we write down explicitly the braided coproduct, antipode, and adjoint action for quantum algebra $\uqsl$. The braided adjoint action is seen to coincide with the ordinary quantum adjoint action, which also follows from the general results of S. Majid.
\end{abstract}


\maketitle

\section{Introduction}

The motivation for the results in this article comes from our attempts in \cite{PS2} to write down a Dirac element $D$ in the tensor product of $U_q(\mathfrak{sl}_n)$ and the appropriate Clifford algebra, at least for $n=3$. The properties we would like $D$ to have are the same ones that its classical analogue has. (See also \cite{PS1} where this was achieved for $n=2$.)

To recall  the classical setting, let $G$ be a real reductive Lie group with a Cartan involution $\Theta$ and the corresponding maximal compact subgroup $K=G^\Theta$. Let $\frg=\frk\oplus\frp$ be the (Cartan) decomposition of the complexified Lie algebra $\frg$ of $G$ into eigenspaces of $\theta=d\Theta$, so that $\frk$ is the complexified Lie algebra of $K$ and $\frp$ is the $(-1)$-eigenspace of $\theta$.

Let $B$ be a non-degenerate invariant symmetric bilinear form on $\frg$ (e.g. the Killing form or the trace form). Then $B$ is nondegenerate on $\frp$; let $C(\frp)$ be the corresponding Clifford algebra. If $b_i$ is a basis of $\frp$ and if $d_i$ is the dual basis, then Parthasarathy \cite{Par} defined the Dirac operator as
\[
D=\sum b_i\otimes d_i\quad\in U(\frg)\otimes C(\frp).
\]
Then $D$ is independent of the choice of the basis $b_i$ and $K$-invariant for the adjoint action on both factors. Moreover, its square is the spin Laplacian
\[
D^2=-\Cas_\frg+\Cas_{\Delta(\frk)}+\|\rho_\frk\|^2-\|\rho_\frg\|^2.
\] 
Here $\Delta:\frk\to U(\frg)\otimes C(\frp)$ is defined on $X\in \frk$ by
\[
\Delta(X)=  X\otimes 1+1\otimes\alpha(X),
\]
where $\alpha:\frk\to C(\frp)$ is defined by
\[
\frk\to\frso(\frp)\cong\twedge^2(\frp)\hookrightarrow C(\frp),
\]
with the first arrow being the action map, and the last arrow being the Chevalley map (also called the quantization map, or the skew symmetrization map).

For more subtle properties of $D$ and its action on (spinorized) Harish-Chandra modules, like Vogan's conjecture, see \cite{V}, \cite{HP1} and \cite{HP2}.

The setting for quantizing the above situation is obtained by replacing $U(\frg)$ by $U_q(\frg)$ and $U(\frk)$ by $U_q(\frk)$, and then defining $\frp$ to be an appropriate subspace of $U_q(\frg)$ invariant under the (quantum) adjoint action of $U_q(\frk)\subset U_q(\frg)$. It turns out that ``quantizing" the Clifford algebra $C(\frp)$ does not change its algebra structure due to its rigidity, and we get a map 
$\alpha_q:U_q(\frk)\to C(\frp)$ since $C(\frp)$ is the endomorphism ring of the spin module $S$, which is a $\frk$-module and thus also 
a $U_q(\frk)$-module.
Similarly, $C(\frp)$ is also a $U_q(\frk)$ module, and the map $\alpha_q$ is a morphism of $U_q(\frk)$-modules.

To obtain a diagonal mapping $\Delta_q:U_q(\frk)\to U_q(\frg)\otimes C(\frp)$ analogous to the above map $\Delta$, a reasonable idea would be to use the coproduct $\triangle:U_q(\frk)\to U_q(\frk)\otimes U_q(\frk)$ followed by $\id\otimes\alpha$. This is however problematic, since the coproduct is not a $U_q(\frk)$-module map for the adjoint action. Also, we can define $D\in U_q(\frg)\otimes C(\frp)$ similarly as in the classical situation, but computing $D^2:=D\circ D$ does not simplify nicely as in the classical case if we use the ordinary tensor product algebra structure on $U_q(\frg)\otimes C(\frp)$.

All of the above problems disappear if we use the braided structures. More precisely, we consider the braided tensor product algebras $U_q(\frg)\underline{\otimes} C(\frp)$ and $U_q(\frk)\underline{\otimes} U_q(\frk)$, and to define the diagonal map $\Delta_q$ we use the braided coproduct 
 $\underline{\triangle}:U_q(\frk)\to U_q(\frk)\underline{\otimes} U_q(\frk)$ instead of the ordinary coproduct $\triangle:U_q(\frk)\to U_q(\frk)\otimes U_q(\frk)$. These definitions are due to S. Majid; see \cite{M5}, and also  \cite{maqb}, \cite{ma}, \cite{mabook} and \cite{maadjoint}.
 
 The ordinary quantum adjoint action is also replaced by its braided version, which however turns out to be the same as the ordinary one. We will demonstrate this phenomenon explicitly, but it also follows from the general results of 
Majid, \cite{maadjoint}, Appendix.
In that article one identifies $B=U_q(sl_n)$ as the vector space with braided group $BU_q(sl_n)$, so that the braided adjoint action 
becomes the usual quantum adjoint action as a Hopf algebra. The proof of this fact explicitly follows from general reconstruction theory arguments, and is presented in \cite{maadjoint} in terms of the adjoint coactions rather than the adjoint action, but one can just turn 
all diagrams upside down for adjoint actions. 
 
 Since our first example is $\frg=\frsl(3,\bbC)$, $\frk=\frgl(2,\bbC)$, and since $\frgl(2,\bbC)$ is very similar to $\frsl(2,\bbC)$, we set out to understand the braided structure for $\frsl(2,\bbC)$ in detail. The result of this endeavor is the present article. Essentially, all the results we obtain are known by the work of Majid. We however believe that having all the formulas (including the ones with infinite $q$-expansion) explicitly written down, with elementary proofs and gathered in one place, will be of sufficient interest to the readers to justify writing this article.

\section{Notation and conventions}

We follow the conventions and notation from \cite{ks}. Let $q$ be a fixed complex number not equal to $0$ or $\pm 1$. We also assume $q$ is not a root of unity.

Let $\uqsl=U_q(\mfrak{sl}(2,{\mathbb C}))$ 
be the associative unital algebra over $\bbC$ generated by 
\[
K, K^{-1}, E, F,
\] 
with relations 
\begin{eqnarray*}
	\label{defuq}
	& KK^{-1}=1=K^{-1}K; \nonumber \\
	& KE=q^2EK,\qquad KF=q^{-2}FK; \\ 
	& EF-FE=\frac{K-K^{-1}}{q-q^{-1}}\nonumber.
\end{eqnarray*}

The Hopf algebra structure on $\uqsl$ is given as follows. The coproduct is the algebra homomorphism
\begin{align*}
	\triangle:\, \uqsl\longrightarrow 
	\uqsl\otimes \uqsl
\end{align*}
given on generators by
\begin{eqnarray*} 
	&\triangle(K^{\pm 1})=K^{\pm 1}\otimes K^{\pm 1},\nonumber\\ 
	&\triangle(E)=E\otimes K+1\otimes E,\\ 
	&\triangle(F)=F\otimes 1+K^{-1}\otimes\nonumber F.
\end{eqnarray*}
We will use the Sweedler notation and write
\[
\triangle(u)=\sum u_{(1)}\otimes u_{(2)}
\]
for any $u\in\uqsl$.

The counit is the algebra homomorphism $\epsilon:\uqsl\to\bbC$ given on generators by
\[
	\epsilon(K)=\epsilon(K^{-1})=1,\qquad \epsilon(E)=\epsilon(F)=0.
\]

The antipode $S$ is the antiautomorphism of the algebra $\uqsl$ given on generators by
\begin{eqnarray*}
	&S(K)=K^{-1},\quad S(K^{-1})=K,\\
	&S(E)=-EK^{-1},\quad S(F)=-KF.\nonumber
\end{eqnarray*}

The quasi-triangular structure on $\uqsl$ is 
given by the $R$-matrix 

\begin{eqnarray}\label{R-matrix}
	& & R=q^{\frac{1}{2}H\otimes H}\sum_{n=0}^\infty q^{\frac{n(n+1)}{2}}\frac{(1-q^{-2})^n}{[n]_q!}(E^n\otimes F^n)
	\nonumber \\
	& & =\sum_{n=0}^\infty\sum_{l=0}^\infty (\frac{ln(q)}{2})^l\frac{1}{l!} q^{\frac{n(n+1)}{2}}\frac{(1-q^{-2})^n}{[n]_q!}(H^lE^n\otimes H^lF^n),
\end{eqnarray}
where $[n]_q!$ is the $q$-factorial of $n\in\mN$, i.e., 
$[n]_q!=[n]_q[n-1]_q\ldots [1]_q$ with $[n]_q=\frac{q^n-q^{-n}}{q-q^{-1}}$, and $[0]_q!=1$.
In particular, the coefficient $q^{\frac{n(n+1)}{2}}\frac{(1-q^{-2})^n}{[n]_q!}$ equals to
$1$, $q-q^{-1}$ and $\frac{q(q-q^{-1})^2}{q+q^{-1}}$ for $n=0$, $n=1$ and $n=2$, respectively.

The $R$-matrix lives in a suitable completion of $\uqsl$, which also contains an element $H=\ln(K)/\ln(q)$ such that $q^H=K$. We can think of $H$ as the standard element of the Cartan subalgebra of the classical $\frsl_2$. It satisfies the usual $\frsl_2$ commutation relations with $E$ and $F$, and also
\[
\triangle(H)=H\otimes 1+1\otimes H;\qquad S(H)=-H.
\]
We denote $R=\sum R_1\otimes R_2 \in \uqsl\otimes \uqsl$
and use the notation $\hat{R}$ for the composition
$T\circ R$, where $T$ is the flip 
$a\otimes b\mapsto b\otimes a$.

For every quasi-triangular Hopf algebra $H$ one can define its braided group analogue $\underline{H}$, 
called the transmutation of ${H}$, cf. \cite{mabook}.
The transmuted 
Hopf algebra $\underline{H}$ has the same algebra structure and the same counit as $H$, but the coproduct
$\triangle$ and antipode $S$ of $H$ are changed to the
braided coproduct $\underline{\triangle}$ and the 
braided antipode $\underline{S}$, respectively:
\begin{eqnarray}\label{bc}
& & \underline{\triangle}(a)=\sum a_{(1)}S(R_2)\otimes R_1\triangleright a_{(2)},
\\ \label{ba}
& & \underline{S}a=\sum R_2S(R_1\triangleright a) .
\end{eqnarray}  
Here $b\triangleright a$ denotes the left (quantum) adjoint action of $b$ on $a$:
\[
b\triangleright a =ad_b(a)=\sum b_{(1)}a S(b_{(2)}).
\]
For a given Hopf algebra $H$, the coproduct 
$\triangle: H\to H\otimes H$ is an algebra map which is however not an 
$H$-module map for the adjoint action $ad$. On the other hand, the braided coproduct 
$\underline{\triangle}: H\to H\underline{\otimes} H$ is an $H$-module algebra map for the braided adjoint action $\underline{ad}$. Here the braided tensor 
product algebra structure on $H\underline{\otimes}H$ (which is equal to $H\otimes H$ as a vector space) is given by   
\begin{eqnarray*}\label{bprod}
	(a_1\underline{\otimes}b_1)\cdot (a_2\underline{\otimes}b_2)=
	a_1\hat{R}(b_1\otimes a_2)b_2, \quad a_1,a_2, b_1,b_2\in H, 
\end{eqnarray*}
and the braided adjoint action of $H$ on itself is given by 
\begin{eqnarray}\label{badjoint}
\underline{ad}_X(Y)=\cdot^2 \circ (Id\otimes \hat{R})\circ (Id\otimes \underline{S}\otimes Id)
\circ (\underline{\triangle}\otimes Id)(X\otimes Y),
\end{eqnarray}  
where $\cdot^2$ denotes 
the multiplication map of the triple tensor product, $\cdot^2(X\otimes Y\otimes Z)= XYZ$,
$X,Y,Z\in H$.

We note that while $\underline{\triangle}$ is a 
homomorphism of algebras, $\underline{S}$ is 
not an anti-homomorphism. Instead, it
satisfies the relation
\begin{eqnarray}\label{braidShom}
\underline{S}(ab)=\cdot^2\circ\hat{R}({\underline S}(a)\otimes{\underline S}(b)), \quad a,b\in H .
\end{eqnarray}

For computations in the following sections we will need the following lemma (cf. \cite{ks}, Section 3.1.1, equation (5)).

\begin{lemma}\label{efn}
Let $n\in\mN$. There is an identity
\begin{eqnarray*}
E^nF-FE^n=q^{n-1}\frac{[n]_q}{q-q^{-1}}E^{n-1}K-q^{-n+1}\frac{[n]_q}{q-q^{-1}}E^{n-1}K^{-1}.
\end{eqnarray*}
\end{lemma}
\begin{proof}
The proof goes by induction on $n$ by multiplying the formula
for $n-1$ by $E$ from the left and applying the relation 
$[E,F]=\frac{K-K^{-1}}{q-q^{-1}}$. 
\end{proof}

\section{The braided coproduct $\underline{\triangle}$}

\begin{theorem}
The braided coproduct of $U_q(sl_2)$ generators is 
given by
\begin{enumerate}
\item 
\begin{eqnarray*}
\underline{\triangle}(E)=E\otimes K+K^{-1}\otimes E + q^{-1}(q-q^{-1})^2EF\otimes E,
\end{eqnarray*}
\item
\begin{eqnarray*}
\underline{\triangle}(K)=K\otimes K +q^{-1}(q-q^{-1})^2KF\otimes E,
\end{eqnarray*}
\item
\begin{eqnarray*}
	\underline{\triangle}(K^{-1})= \sum_{n=0}^\infty (-1)^{n}q^{-n^2-2n}(q-q^{-1})^{2n} K^{-1}F^n\otimes E^{n}K^{-n-1},
\end{eqnarray*}
\item
\begin{eqnarray*}\label{bcF}
\underline{\triangle}(F)=1\otimes F + \sum_{n=1}^\infty (-1)^{n+1}q^{-n^2+1}(q-q^{-1})^{2n-2} F^n\otimes E^{n-1}K^{-n-1}.
\end{eqnarray*}
\end{enumerate}
\end{theorem}
\begin{proof}
	(1) Since $\triangle(E)=E\otimes K + 1\otimes E$, and since $ad_E(E)=0$, $ad_E(K)=(1-q^2)E$, the only non-trivial
	contributions to the $R$-matrix \eqref{R-matrix} in this case correspond 
	to $n=0,1$. We easily compute 
	\begin{eqnarray} \label{adE}
		& ad_{H^l}(K)=\delta_{l,0}K, \quad ad_{H^lE}(K)=2^l(1-q^2)E, 
		\nonumber \\
		& \quad ad_{H^lE^n}(K)=0 \quad \mbox{for $n\geq 2$}, \quad ad_{H^l}(E)=2^l E, 
	\end{eqnarray}
and substitute into \eqref{bc} to get the claim.
	
(2) Since $ad_{E}(K)=(1-q^2)E$ (hence 
$ad_{E^2}(K)=0$), the only non-trivial
contributions to the $R$-matrix \eqref{R-matrix} in this case correspond 
to $n=0,1$. In particular, an easy computation gives the result.

(3) Since $\triangle(K^{-1})=K^{-1}\otimes K^{-1}$, we have to compute
\[
\sum K^{-1}S(R_2)\otimes R_1\triangleright K^{-1},
\]
where $R$ is given by \eqref{R-matrix}. 
One sees
\[
S(H^lF^n)=(-1)^{n+l}(KF)^nH^l,
\]
and
\begin{equation}
	\label{En on Kinv}
(H^lE^n)\triangleright K^{-1}=(2n)^l(1-q^{-2n})(1-q^{-2(n-1)})\ldots(1-q^{-2})E^nK^{-n-1}.
\end{equation}
Noting that
\[
\sum_{l=0}^\infty \left(\frac{\ln q}{2}\right)^l\frac{1}{l!}(-1)^l(2n)^l H^l=K^{-n}
\]
and $(KF)^nK^{-n}=q^{-n(n+1)}F^n$, we arrive at the result after simplifying the coefficients.

(4) This formula can be obtained either by a similar computation as above, or by making a finite computation to see
\begin{eqnarray*}
	\underline{\triangle}(KF)=K\otimes KF +KF\otimes K^{-1}+q^{-1}(q-q^{-1})^2KF\otimes EF,
\end{eqnarray*}
and then compute  $\underline\triangle(F)=\underline\triangle(K^{-1})\underline\triangle(KF)$ using the formula (3).
\end{proof}

\section{The braided antipode $\underline{S}$}

Recall that the quantum Casimir element of $U_q(sl_2)$ is defined as
\begin{eqnarray}
	\label{qcas}
	Cas_q:=EF+\frac{qK^{-1}+q^{-1}K}{(q-q^{-1})^2}\, \in Z(U_q(sl_2)).
\end{eqnarray}
$Cas_q$ is a generator of the free algebra $Z(U_q(sl_2))$, the center 
of $U_q(sl_2)$. Notice that $Cas_q$ can be modified by a 
rational function of $q$ to make it regular in the classical 
limit $q\to 1$.

\begin{theorem}
\label{thm antipode}
The braided antipode, cf. \eqref{ba}, of $U_q(sl_2)$ generators is 
given by
\begin{enumerate}
	\item
	\begin{eqnarray*}
		\underline{S}(E)=-q^{2}E,
	\end{eqnarray*}
\item
\begin{eqnarray*}\label{Kantipode}
	\underline{S}(K)=K^{-1}+q(q-q^{-1})^{2}FE=-q^2K+q(q-q^{-1})^{2}Cas_q,
\end{eqnarray*}
\item
\begin{eqnarray*}\label{K-1antipode}
\underline{S}(K^{-1})=\sum_{n=0}^\infty (-1)^{n}q^{n^2+2n}(q-q^{-1})^{2n} K^{n+1}F^{n}E^{n}.
\end{eqnarray*}
\item 
\begin{eqnarray*}
\underline{S}(F)=\sum_{n=0}^\infty (-1)^{n+1}q^{n^2+4n+4}(q-q^{-1})^{2n} K^{n+2}F^{n+1}E^{n},
\end{eqnarray*}
\end{enumerate}
\end{theorem}

\begin{proof}
	(1) The claim follows easily from $ad_E(E)=0$. 
	
	(2) Since $ad_{H^l}(K)=0$ for all
	non-zero $l\in\mN$ and $ad_{H^lE}(K)=2^l(1-q^2)E$, it follows that
	\begin{eqnarray*}
		\underline{S}(K)=S(K) + (q-q^{-1})(1-q^2)KF(-E)K^{-1}
		=K^{-1}+ q(q-q^{-1})^2FE.
	\end{eqnarray*}
	
(3) We use the formula \eqref{En on Kinv} to get
\begin{eqnarray*}
S((H^lE^n)\triangleright K^{-1})=(-1)^{n}(2n)^lq^{n^2+n}(1-q^{-2n})(1-q^{-2(n-1)})\ldots (1-q^{-2})KE^{n}.
\end{eqnarray*}
The claim now follows by summing over $l$ in \eqref{R-matrix}.

(4) To compute $\underline{S}(F)$,
we first compute
\begin{eqnarray*}
	\underline{S}(KF)=-q^{2}KF
\end{eqnarray*}
similarly as above. Then $\underline{S}(F)$ can be expressed as 
\begin{eqnarray} 
& & \underline{S}(F)=\underline{S}(K^{-1}KF)=\cdot^2\circ \hat{R}(\underline{S}(K^{-1})\otimes \underline{S}(KF))
\nonumber \\ \nonumber
 & & =-q^2\cdot^2\circ \hat{R}(\underline{S}(K^{-1})\otimes KF) . 
\end{eqnarray}
Because $F\triangleright(KF)$=0, 
we get $\hat{R}(\underline{S}(K^{-1})\otimes KF)=\underline{S}(K^{-1})\otimes KF$,
and so 
\begin{eqnarray}
& & \underline{S}(F)=-q^2\cdot^2\circ \hat{R}(\underline{S}(K^{-1})\otimes KF)
\nonumber \\ \nonumber
& & =-q^2\sum_{n=0}^\infty (-1)^{n}q^{n^2+2n}(q-q^{-1})^{2n} KFK^{n+1}F^{n}E^{n}  
\nonumber \\
& & =\sum_{n=0}^\infty (-1)^{n+1}q^{n^2+4n+4}(q-q^{-1})^{2n} K^{n+2}F^{n+1}E^{n} . \nonumber 
\end{eqnarray} 
\end{proof}

The reader is invited to directly verify the structural relations 
in $U_q(sl_2)$ by applying \eqref{braidShom}, e.g. $KEK^{-1}=q^2E$; we give one such example in the following lemma.

\begin{lemma}
The formula \eqref{braidShom} implies 
\begin{eqnarray}\label{1KKinverse}
1=\cdot^2\circ\hat{R}(\underline{S}(K)\otimes \underline{S}(K^{-1})).
\end{eqnarray}
\end{lemma}
\begin{proof}
As we know $\underline{S}(K)$, cf. Theorem \ref{thm antipode},
the formulas $ad_E(K)=(1-q^2)E, ad_{E^2}(K)=0, ad_E(Cas_q)=0$ imply
\begin{eqnarray}
& & R(-q^2K+q(q-q^{-1})^2Cas_q,\underline{S}(K^{-1}))= 
\nonumber \\
& & (-q^2K+q(q-q^{-1})^2Cas_q)\otimes\underline{S}(K^{-1}) 
\nonumber \\
& &
+q^{-2}(q-q^{-1})(-q^2(1-q^2)E\otimes F\triangleright\underline{S}(K^{-1})).
\end{eqnarray}
The key computation is then based on the left adjoint action of 
$F$ and Lemma \ref{efn}:
\begin{eqnarray}
& & F\triangleright K^{n+1}F^nE^n = 
(q^{2n+2}-1)K^{n+1}F^{n+1}E^{n}
\\ \nonumber
& & -q^{n+1}\frac{q^n-q^{-n}}{(q-q^{-1})^2}K^{n+2}F^{n}E^{n-1}
+q^{-n-1}\frac{q^n-q^{-n}}{(q-q^{-1})^2}K^{n}F^{n}E^{n-1}
\end{eqnarray}
Therefore \eqref{1KKinverse} is equivalent, based on  
$\underline{S}(K^{-1})$, cf. \eqref{K-1antipode}, to
\begin{eqnarray}
& & 1=\sum_{n=0}^\infty (-1)^{n+1}q^{n^2+2n+2}(q-q^{-1})^{2n} K^{n+2}F^nE^{n}
\nonumber \\
& & +\sum_{n=0}^\infty (-1)^{n}q^{n^2+2n+1}(q-q^{-1})^{2n+2} K^{n+1}F^nE^{n}(FE+\frac{qK+q^{-1}K^{-1}}{(q-q^{-1})^2})
\nonumber \\
& & +\sum_{n=0}^\infty (-1)^{n}q^{n^2+2n+1}(q-q^{-1})^{2n+2}(q^{2n+2}-1) K^{n+1}F^{n+1}E^{n+1}
\nonumber \\
& & +\sum_{n=0}^\infty (-1)^{n+1}q^{n^2+3n+2}(q-q^{-1})^{2n}(q^{n}-q^{-n}) K^{n+2}F^{n}E^{n}
\nonumber \\
& & +\sum_{n=0}^\infty (-1)^{n}q^{n^2+n}(q-q^{-1})^{2n}(q^{n}-q^{-n}) K^{n}F^{n}E^{n} .
\end{eqnarray}
We use Lemma \ref{efn} in the second sum above, and then 
observe that all contributions to the monomials $K^{n+2}F^{n}E^{n}$
collect to zero as well as contributions to the monomials
$K^{n}F^{n}E^{n}$ collect to zero with the only 
exception, $n=0$. This then compares with the left hand 
side of the last equality, and the proof is complete.
\end{proof}

As a demonstration of the use of our formulas, we compute the 
braided coproduct and the braided antipode of the quantum Casimir 
operator. 

\begin{lemma}
	\label{lem qcas}
Let $Cas_q$ be the quantum Casimir 
operator defined in \eqref{qcas}. Then
\begin{eqnarray} 
	& & \underline\triangle{Cas_q}= E\otimes KF + q^{-2}KF\otimes E +q^{-1}(q-q^{-1})^2Cas_q\otimes Cas_q
	\nonumber \\
	& & -q^{-2}Cas_q\otimes K -q^{-2}K\otimes Cas_q + q^{-2}\frac{q+q^{-1}}{(q-q^{-1})^2}{K}\otimes{K},
\end{eqnarray}
and
\begin{eqnarray}
	& & \underline{S}(Cas_q) =  Cas_q .
\end{eqnarray}

\end{lemma}
\begin{proof}
We have 
\begin{eqnarray}
\underline\triangle{Cas_q}=\underline\triangle{E}\cdot\underline\triangle{F}+
\frac{q}{{(q-q^{-1})^2}}\underline\triangle{K^{-1}}+\frac{q^{-1}}{(q-q^{-1})^2}\underline\triangle{K}
\end{eqnarray}
with $\underline\triangle{E}\cdot\underline\triangle{F}$ the braided
tensor product in $U_q(sl_2)\underline{\otimes}U_q(sl_2)$.
The following calculation is based on the use of
\begin{eqnarray}
ad_E(K)=(1-q^2)E, \quad
ad_F(F^n)=(1-q^{2n})F^{n+1}
\end{eqnarray}
together with the R-matrix \eqref{R-matrix}:
\begin{eqnarray}
& & R(K\otimes 1)=K\otimes 1,
\nonumber \\
& & R(K\otimes F^n)=K\otimes F^n +(q-q^{-1})(1-q^{2})(1-q^{2n})q^{-2n-2}E\otimes F^{n+1},
\nonumber \\
& & R(E\otimes 1)=E\otimes 1,
\nonumber \\
& & R(E\otimes F^n)=q^{-2n}E\otimes F^n
\end{eqnarray}
for all $n\in\mN$. Then the braided tensor product of 
$\underline\triangle{E}$ and $\underline\triangle{F}$
is 
\begin{eqnarray}
& & \underline\triangle{E}\cdot\underline\triangle{F}=
\nonumber \\
& & = E\otimes KF
\nonumber \\
& & +\sum_{n=1}^\infty (-1)^{n+1}q^{-n^2+1}(q-q^{-1})^{2n-2}EF^n\otimes KE^{n-1}K^{-n-1}
\nonumber \\
& & +\sum_{n=1}^\infty (-1)^{n}q^{-n^2-2n}(q-q^{-1})^{2n}(1-q^{2n})EF^{n+1}\otimes E^{n}K^{-n-1}
\nonumber \\
& & +K^{-1}\otimes EF
\nonumber \\
& & +\sum_{n=1}^\infty (-1)^{n+1}q^{-n^2-2n+1}(q-q^{-1})^{2n-2}K^{-1}F^n\otimes E^{n}K^{-n-1}
\nonumber \\
& & +q^{-1}(q-q^{-1})^2EF\otimes EF
\nonumber \\
& & +\sum_{n=1}^\infty (-1)^{n+1}q^{-n^2-2n}(q-q^{-1})^{2n}EF^{n+1}\otimes E^{n}K^{-n-1} .
\end{eqnarray}
The first, second and fourth infinite sum cancel out, while the third cancels 
with the contribution of $\underline\triangle{K^{-1}}$ to the Casimir element 
$Cas_q$. In conclusion,
\begin{eqnarray} 
& & \underline\triangle{Cas_q}=E\otimes KF + EF\otimes K^{-1} 
\nonumber \\
& & + K^{-1}\otimes EF +q^{-1}(q-q^{-1})^2EF\otimes EF + q^{-2}KF\otimes E 
\nonumber \\
& & + \frac{q}{{(q-q^{-1})^2}}{K^{-1}}\otimes{K^{-1}}
+ \frac{q^{-1}}{{(q-q^{-1})^2}}{K}\otimes{K}
\nonumber \\
& & = E\otimes KF + q^{-2}KF\otimes E +q^{-1}(q-q^{-1})^2Cas_q\otimes Cas_q
\nonumber \\
& & -q^{-2}Cas_q\otimes K -q^{-2}K\otimes Cas_q + q^{-2}\frac{q+q^{-1}}{(q-q^{-1})^2}{K}\otimes{K} .
\end{eqnarray}
As for the action of braided antipode, we have 
\begin{eqnarray}
& & \underline{S}(Cas_q)=\underline{S}(EF)+\frac{q}{(q-q^{-1})^2}\underline{S}(K^{-1})
+\frac{q^{-1}}{(q-q^{-1})^2}\underline{S}(K)
\nonumber \\
& & = \cdot^2 \circ \hat{R} (\underline{S}(E)\otimes\underline{S}(F))
\nonumber \\
& & +\frac{q}{(q-q^{-1})^2}\sum_{n=0}^\infty (-1)^{n}q^{n^2+2n}(q-q^{-1})^{2n}K^{n+1}F^{n}E^{n}
\nonumber \\
& & +\frac{q^{-1}}{(q-q^{-1})^2}(q(q-q^{-1})^2Cas_q-q^2K)
\end{eqnarray}
and because the first terms evaluates to $-\underline{S}(F)E$,
we get
\begin{eqnarray}
& & = \sum_{n=0}^\infty (-1)^{n}q^{n^2+4n+4}(q-q^{-1})^{2n}K^{n+2}F^{n+1}E^{n+1}
\nonumber \\
& & + \sum_{n=0}^\infty (-1)^{n}q^{n^2+2n+1}(q-q^{-1})^{2n-2}K^{n+1}F^{n}E^{n}
\nonumber \\
& & +\, Cas_q-\frac{q}{(q-q^{-1})^2}K = Cas_q .
\end{eqnarray}
\end{proof}

\section{The braided adjoint action $\underline{ad}$}

We recall that the elements $E, FE-q^{-2}EF, KF$ span a subrepresentation of the adjoint representation 
of $U_q(sl_2)$, cf. \cite{bu}:
\begin{eqnarray}
& & ad_F(E)=FE-q^{-2}EF, \quad ad_F(FE-q^{-2}EF)=-(q+q^{-1})KF,
\nonumber \\
& & ad_E(KF)=-(FE-q^{-2}EF), \quad ad_E(FE-q^{-2}EF)=(q+q^{-1})E .
\end{eqnarray}

The following proposition is a special case of a general result of S. Majid \cite{maadjoint}.

\begin{proposition}
The braided adjoint action of $U_q(sl_2)$ on itself is the same as the ordinary quantum adjoint action. In other words, for any $X\in U_q(sl_2)$, 
\begin{enumerate}
\item $\underline{ad}_{E}(X)=(EX-XE)K^{-1}$,
\item $\underline{ad}_{K}(X)=KXK^{-1}$,
\item $\underline{ad}_{K^{-1}}(X)=K^{-1}XK$,
\item $\underline{ad}_F(X)=FX-K^{-1}XKF$.
\end{enumerate}
\end{proposition}
\begin{proof}
(1) By \eqref{badjoint}, we have
\begin{eqnarray*}
& & \underline{ad}_E(X)=\cdot^2 \circ (Id\otimes \hat{R})\circ (
E\otimes K^{-1} \otimes X + q(q-q^{-1})^{2}E\otimes FE\otimes X
\nonumber \\
& & -q^2 K^{-1}\otimes E\otimes X - q(q-q^{-1})^{2}EF\otimes E\otimes X)
\end{eqnarray*}
The middle components in the last two terms are trivial for the 
action of $E$, $ad_E(E)=0$. As for the middle terms of the first
two terms, we have
\begin{eqnarray}\label{adEsquare}
\nonumber ad_E(K^{-1} + q(q-q^{-1})^{2} FE) = \\ \nonumber
(1-q^{-2})EK^{-2}
+ q(q-q^{-1})^{2}\frac{q^2}{q-q^{-1}}(E-K^{-2}E) = \\ 
q^3(q-q^{-1})^{2}E,
\end{eqnarray}
and in particular $(ad_E)^2(K^{-1} + q(q-q^{-1})^{2} FE)=0$.
Then the action of the $R$-matrix results in
\begin{eqnarray*}
& & \underline{ad}_E(X)=EXK^{-1} + q(q-q^{-1})^{2}EXFE
+ q^3(q-q^{-1})^{2}E\,ad_{KF}(X)E 
\nonumber \\
& & - q^2K^{-1}ad_{K}(X)E - q(q-q^{-1})^{2}EFad_K(X)E . 
\end{eqnarray*}  
The substitition of $ad_K(X)=KXK^{-1}$ and $ad_{KF}(X)=q^{-2}FKXK^{-1}-q^{-2}XF$
results in
\begin{eqnarray*}
& & \underline{ad}_E(X)=EXK^{-1} + q(q-q^{-1})^{2}EXFE
- q(q-q^{-1})^{2}EXFE - q^2XK^{-1}E 
\nonumber \\
& & = (EX-XE)K^{-1} \nonumber 
\end{eqnarray*}
which completes the proof of (1).

(2) We again use  \eqref{badjoint} to conclude
\begin{eqnarray*}
& & \underline{ad}_K(X)=\cdot^2 \circ (Id\otimes \hat{R})\circ (
K\otimes K^{-1} \otimes X + q(q-q^{-1})^{2}K\otimes FE\otimes X
\nonumber \\
& & - q(q-q^{-1})^{2}KF\otimes E\otimes X) .
\end{eqnarray*}
The first two contributions combine together, and we use \eqref{adEsquare} 
when we apply $\hat{R}$ to $K^{-1} + q(q-q^{-1})^{2} FE$. The action of
$\hat{R}$ on the third term is
easy to see, because we apply $ad_E$ to $E$ again. Altogether,
\begin{eqnarray}
& & \underline{ad}_K(X)= K XK^{-1} + q(q-q^{-1})^{2} KXFE + q(q-q^{-1})^{2} KFKXK^{-1}E
\nonumber \\
& & - q(q-q^{-1})^{2}KXFE  - q(q-q^{-1})^{2}KFKXK^{-1}E
\end{eqnarray}
and the last four contributions do cancel out to yield
\begin{eqnarray*}
& & \underline{ad}_K(X)= K XK^{-1}={ad}_K(X).
\end{eqnarray*}

(3) 
Since $u\mapsto \underline{ad}_u$ is a representation, it follows from (2) that $\underline{ad}_{K^{-1}}=(\underline{ad}_K)^{-1}$ is conjugation by $K^{-1}$.

(4) We first compute $\underline{ad}_{KF}(X)$; then it is easy to get $\underline{ad}_{F}(X)$ as it is equal to $\underline{ad}_{K{-1}}(\underline{ad}_{KF}(X))$. We assume, as we may, that $ad_K(X)=q^kX$ for some $k\in\mathbb{Z}$.

Since 
\[
ad_F(X)=FX-K^{-1}XKF=FX-q^{-k}XF
\]
and
\[
 ad_{F^2}(X)F^2X-q^{-k}(q^2+1)FXF+q^{2-2k}XF^2,
\]
we have
\begin{eqnarray*}
& & \underline{ad}_{KF}(X)=-q^{2-2k}K^2XF 
\nonumber \\
& & + q^2(q-q^{-1})(KFX-q^{-k}KXF)(q^{-1}(q-q^{-1})Cas_q-q^{-1}\frac{q+q^{-1}}{q-q^{-1}}K)
\nonumber \\
& & + q^{k-1}(q-q^{-1})^2(KF^2XE-q^{-k}(q^2+1)KFXFE+q^{2-2k}KXF^{2}E)
\nonumber \\
& & +q^{-k+2}K^2FX-(q-q^{-1})^2q^{k-1}KF^2XE+KFXFE
\nonumber \\
& & +q^{-1}(q-q^{-1})^2KFXFE=q^{k-2}FX-q^{-2}XF .
\end{eqnarray*}
Hence
\begin{eqnarray*}
	& & \underline{ad}_{F}(X)=\underline{ad}_{K^{-1}}(\underline{ad}_{KF}(X))=K^{-1}(q^{k-2}FX-q^{-2}XF)K\\
	& & =FX-q^{-k}XF=ad_F(X).
\end{eqnarray*}
\end{proof}

We remark that the straightforward computation of $\underline{ad}_{K^{-1}}$
based on \eqref{badjoint} is rather lengthy and tedious, cf. the formula for
$\underline{\triangle}(K^{-1})$ from which it follows that the braided 
antipode $\underline{S}$ is applied to a combination of monomials $E^nK^{-n-1}$, $n\in\mN$, and this 
yields another nested infinite sum. The readers are invited to work this 
line of reasoning on their own.


\end{document}